\documentclass{amsart}
\usepackage{amssymb, amsmath, latexsym, mathabx, bm, dsfont}

\renewcommand{\baselinestretch}{\baselinestretch}
\renewcommand{\baselinestretch}{1.1}
\numberwithin{equation}{section}

\newtheorem{thm}{Theorem}[section]
\newtheorem{lem}[thm]{Lemma}
\newtheorem{cor}[thm]{Corollary}

\newtheorem{defn}[thm]{Definition}
\newtheorem{rmk}[thm]{Remark}

\newcommand{\ra}{{\ \rightarrow\ }}
\newcommand{\nra}{{\ \nrightarrow\ }}
\newcommand{\gen}{\text{gen}}

\newcommand{\z}{{\mathbb Z}}
\newcommand{\q}{{\mathbb Q}}

\newcommand{\rank}{\text{rank}}

\newcommand{\bx}{\bm x}
\newcommand{\by}{\bm y}
\newcommand{\bz}{\bm z}
\newcommand{\be}{\bm e}
\newcommand{\bv}{\bm v}
\newcommand{\bu}{\bm u}
\newcommand{\bd}{\bm d}

\newcommand{\lglue}{\bm[\!\![}
\newcommand{\rglue}{]\!\!\bm]}

\newcommand{\Mod}[1]{\ (\mathrm{mod}\ #1)}

\begin{document}

\title{Isolations of cubic lattices from their proper sublattices}
\author{Byeong-Kweon Oh}
\address{Department of Mathematical Sciences and Research Institute of Mathematics, Seoul National University, Seoul 08826, Korea}
\email{bkoh@snu.ac.kr}
\thanks{This work was supported by the National Research Foundation of Korea (NRF-2019R1A2C1086347) and (NRF-2020R1A5A1016126).}


\begin{abstract}   A (positive definite and integral) quadratic form is called  {\it an isolation} of a quadratic form $f$ if it represents all subforms of $f$  except for $f$ itself. 
 The minimum rank of isolations of a quadratic form $f$ is denoted, if it exists, by $\text{Iso}(f)$. In this article, we show that $\text{Iso}(I_2)=5$ and  $\text{Iso}(I_3)=6$, where $I_n=x_1^2+\dots+x_n^2$ is the sum of $n$ squares for any positive integer $n$.  After proving that there always exists an isolation of $I_n$ for any positive integer $n$, we provide some explicit lower and upper bounds for $\text{Iso}(I_n)$.   In particular, we show that $\text{Iso}(I_n) \in \Omega(n^{\frac32-\epsilon})$ for any $\epsilon>0$.

\end{abstract}
\maketitle

\section{Introduction}

For a positive integer $n$, an integral quadratic form $f$ of rank $n$ is a homogeneous quadratic polynomial 
$$
f(x_1,x_2,\dots,x_n)=\sum_{i,j=1}^n f_{ij}x_ix_j \quad (f_{ij}=f_{ji} \in \z)
$$
with $n$ variables such that the discriminant $\det(f_{ij})$  is nonzero.  The symmetric matrix $(f_{ij})$ is called the Gram matrix corresponding to the quadratic form $f$. 
Throughout this article, we always assume that any quadratic form $f$ is {\it integral and positive definite}, that is, the corresponding Gram matrix is integral and positive definite. 
  We say  a quadratic form $g(y_1,y_2,\dots,y_m)=\sum_{i,j=1}^m g_{ij}y_iy_j$ of rank $m$ is represented by the form $f$ if there are  integers $t_{ij}$'s such that 
$$
f(t_{11}y_1+t_{12}y_2+\dots+t_{1m}y_m,\dots,t_{n1}y_1+\dots+t_{nm}y_m)=g(y_1,y_2,\dots,y_m).
$$
If $M_f=(f_{ij})$ and $M_g=(g_{ij})$ are the  Gram matrices corresponding to $f$ and $g$, respectively, then $g$ is represented by $f$ if and only if there is an integral matrix  $T=(t_{ij}) \in M_{n,m}(\z)$ such that 
$$
T^tM_fT=M_g.
$$ 
Hence the existence of a representation between two quadratic forms is equivalent to the existence of an integral solution of the system of  diophantine equations given by those quadratic forms. Any quadratic form that is represented by a quadratic form $f$ is called a subform of $f$. 

Clearly, every subform of a quadratic form $g$ is represented by a quadratic form $f$ if $g$ itself is represented by $f$. One may naturally ask whether or not the converse of the above statement is also true, that is, 
\vskip 0.3cm
{\it if every proper subform of $g$ is represented by $f$, then  $g$  is represented by $f$?}   
\vskip 0.3cm
\noindent Related with this question,   it was proved in \cite{jko} that the ternary diagonal quadratic form $f=2x^2+2y^2+5z^2$ represents all squares of integers except for $1$, that is,   $f$ represents all subforms of the unary quadratic form $x^2$ except for $x^2$ itself. On the contrary,  Elkies, Kane, and Kominers proved in \cite{ekk} that any quadratic form which represents all proper subforms of $x^2+y^2+2z^2$ represents $x^2+y^2+2z^2$ itself.  

To study the above question, the following definition seems to be quite natural. A (positive definite and integral) quadratic form is called  {\it an isolation} of a quadratic form $f$ if it represents all subforms  of $f$ except for $f$ itself.  The minimum rank of isolations of a quadratic form $f$ is denoted, if it exists, by $\text{Iso}(f)$.  As stated above, the diagonal ternary quadratic form $2x^2+2y^2+5z^2$ is an isolation of $x^2$, and one may easily check that $\text{Iso}(x^2)=3$.    
 In fact, the existence of an isolation of a quadratic form $f$ is closely related with the uniqueness of a minimal $\mathcal S_f$-universality criterion set, where $\mathcal S_f$ is the set of all  subforms of $f$.

Let $\mathcal S$ be a set of (positive definite and integral) quadratic forms with bounded rank. A quadratic form $f$ is called {\it $\mathcal S$-universal} if it represents all quadratic forms in the set $\mathcal S$. A subset $\mathcal S_0$ of $\mathcal S$ is called an {\it $\mathcal S$-universality criterion set} if any $\mathcal S_0$-universal quadratic form is, in fact, $\mathcal S$-universal. We say $\mathcal S_0$ is {\it minimal} if any proper subset of $\mathcal S_0$ is not an $\mathcal S$-universality criterion set.  The 15-theorem, proved by Conway and Schneeberger in 1993 (see \cite{b}), states that the set $\{1,2,3,5,6,7,10,14,15\}$ is a minimal $\z^+$-universality criterion set, where $\z^+$ denotes the set of all positive integers. Here, a positive integer $a$ corresponds to the unary quadratic form $ax^2$. 

As a generalization of the $15$-Theorem, Kim, Kim, and the author proved in \cite{kko} that there is always a {\it finite} $\mathcal S$-universality criterion set.   After proving that,  the authors asked whether or not  the minimal $\mathcal S$-universality criterion set is unique for any set $\mathcal S$ of quadratic forms with bounded rank. 
In \cite{ekk}, Elkies, Kane, and Kominers answered this question  in the negative by giving simple examples of sets $\mathcal S$ that have minimal $\mathcal S$-universality criteria sets with multiple cardinalities.  In fact, they proved that if $\mathcal S$ is the set of all quadratic subforms of $x^2+y^2+2z^2$, then both $\{x^2
+y^2+2z^2\}$ and $\{x^2+y^2, 2x^2+2y^2+2z^2\}$ are minimal $\mathcal S$-universality criteria sets.  

For a quadratic form $f$, let $\mathcal S_f$ be the set of all subforms of $f$. Clearly, $\{f\}$ is a minimal $\mathcal S_f$-universality criterion set. Suppose that $\{f_1,f_2,\dots,f_t\}$ is another minimal $\mathcal S_f$-universality criterion set. From the definition, $f$ is not isometric to $f_i$ for any $i=1,2,\dots,t$. If there is a quadratic form, say $F$, which represents all proper subforms  $f$, then $F$ represents $f_i$ for any $i=1,2,\dots,t$. From the definition of an $\mathcal S_f$-universality criterion set, $F$ represents all subforms of $f$. In particular, $F$ represents $f$ itself. Therefore there does not exist an isolation of $f$.  Conversely,  suppose that $\{f\}$ is the unique minimal $\mathcal S_f$-universality criterion set.  Let $\widetilde{\mathcal S_f}$ be the set of all proper subforms of $f$, and let $\{f_1,f_2,\dots,f_s\}$ be a minimal $\widetilde{\mathcal S_f}$-universality criterion set. Note that such a finite set exists always by the result of \cite{kko}.  Since $\{f_1,f_2,\dots,f_s\}$ is not an $\mathcal S_f$-universality criterion set, there is a quadratic form which represents $f_i$ for any $i=1,2,\dots,s$, and hence represents all proper subforms of $f$, whereas it does not represent $f$ itself. This implies that there is an isolation of $f$.  Therefore there is an isolation of a quadratic form $f$ if and only if the set $\{f\}$ is the unique minimal $\mathcal S_f$-universality criterion set.

In this article, we prove that  there is an isolation of the quadratic form $I_n=x_1^2+x_2^2+\dots+x_n^2$  whose Gram matrix is the $n\times n$ identity matrix for any positive integer $n$.  Furthermore, we prove that 
$$
\text{Iso}(I_2)=5 \quad \text{and} \quad   \text{Iso}(I_3)=6.
$$ 
In Sections 4 and 5, we  provide  explicit lower and upper  bounds for $\text{Iso}(I_n)$ for any  positive integer $n$. In particular, we show that $\text{Iso}(I_n) \in \Omega(n^{\frac32-\epsilon})$ for any $\epsilon>0$.    Recall that for two arithmetic functions $f(n)$ and $g(n)$, we say $f(n) \in \Omega(g(n))$ if and only if there is a constant $C>0$ such that $C\cdot g(n) \le f(n)$ for any sufficiently large integer $n$.

The subsequent  discussion will be conducted in the language of quadratic spaces and lattices.  The readers are referred to \cite{ki} and \cite{om} for any unexplained notations and terminologies.    For simplicity, the quadratic map and its associated bilinear form on any quadratic space will be denoted by $Q$ and $B$, respectively.  The term {\em lattice} always means a finitely generated $\z$-module on a finite dimensional positive definite quadratic space over $\q$.  

Let $L=\z\bx_1+\z\bx_2+\dots+\z\bx_n$ be a $\z$-lattice of rank $n$. For a prime $p$, let $\z_p$ be the $p$-adic integer ring. We define $L_p=L\otimes \z_p$, which is considered as  a $\z_p$-lattice.  A $\z$-lattice $M$ is said to be represented by  $L$ if there is a linear map $\sigma: M \longrightarrow L$ such that $Q(\sigma(\bx)) = Q(\bx)$ for any $\bx \in M$.  Such a map is called a representation from $M$ into $L$, which is necessarily injective because the symmetric bilinear map defined on $M$ is assumed to be nondegenerate.   If there is a linear map $\sigma_p : M_p \ra L_p$ satisfying the above property for some prime $p$, then we say $M$ is represented by $L$ over $\z_p$. We say $M$ is locally represented by $L$ if $M$ is represented by $L$ over $\z_p$ for any prime $p$.  If $M$ is represented by $L$, then we simply write $M \ra L$. In particular, if $M=\langle m\rangle$ is a unary $\z$-lattice, then we write $m \ra L$ as well as $\langle m\rangle \ra L$.  Two $\z$-lattices $L$ and $M$ are isometric if there exists a representation sending $L$ onto $M$.  In this case we will write $L \cong M$.  If $L$ is a lattice and $A$ is one of its Gram matrix, we will write $L \cong A$. We will often address a positive definite symmetric matrix as a lattice.  If $M$ is isometric to $L$ over $\z_p$ for any prime $p$, then we say $M$ is contained in the genus of $L$, and we write $M \in \gen(L)$.  The number of isometry classes in the  genus of $L$ is called the class number of $L$, and is denoted by $h(L)$. It is well known that the class number of any $\z$-lattice is always finite. It is also well known that a $\z$-lattice $K$ is locally represented by $L$ if and only if there is a $\z$-lattice $M \in \gen(L)$ such that $K \ra M$.  

 The diagonal matrix with entries $a_1, \ldots, a_n$ on its main diagonal will be denoted by $\langle a_1, \ldots, a_n\rangle$.  If $L$ and $M$ are $\z$-lattices, their orthogonal sum is denoted by $L \perp M$. The $\z$-lattice $I_n=\z \bm e_1+\z \bm e_2+\dots+\z \bm e_n$ of rank $n$ whose Gram matrix is the identity matrix is called the {\it cubic lattice} of rank $n$. Hence we have $I_n \simeq \langle 1,1,\dots,1\rangle$. The $\z$-lattice $I_n$ is frequently called the sum of $n$ squares.

\section{Isolations of  $\z$-lattices}

In this section, we introduce some notions and basic facts on $\z$-lattices which are used throughout this article. 

\begin{defn} A $\z$-lattice  $\ell$ is called irrecoverable (by its sublattices) if there is a $\z$-lattice which represents all  sublattices of $\ell$ except for $\ell$ itself. Such a $\z$-lattice is called  an isolation of $\ell$. We say $\ell$ is recoverable if there does not exist such a $\z$-lattice satisfying the above property. 
\end{defn} 

Note that a $\z$-lattice $\ell$ is irrecoverable if and only if there is an isolation of $\ell$. 
It was proved in \cite{jko} that the unary cubic $\z$-lattice $I_1=\langle 1\rangle$ is irrecoverable and hence any unary $\z$-lattice is irrecoverable.   
Furthermore, it was proved that there are exactly $15$ ternary diagonal isolations of $I_1$. For example, the ternary diagonal $\z$-lattice $\langle2,2,5\rangle$ represents all squares of integers except for $1$. A recoverable $\z$-lattice was first given in \cite{ekk} by Elkies, Kane, and Kominers. They proved that the ternary $\z$-lattice $\langle 1,1,2\rangle$ is recoverable, which answers a question of Kim, Kim and the author \cite{kko} in the negative.  For some recent development on  binary irrecoverable $\z$-lattices, see \cite{klo}.  In fact, there are infinitely many recoverable binary $\z$-lattices up to isometry including $\langle 1,4\rangle$. 

\begin{lem} The cubic $\z$-lattice $I_n$ is irrecoverable for any  positive integer $n$. 
\end{lem}

\begin{proof}  Let $\Phi(I_n)$ be the set of all proper sublattices of $I_n$. Then by the result of \cite{kko}, there is a finite subset $\Phi^0(I_n)=\{ \ell_1,\ell_2,\dots,\ell_t\}$ of 
$\Phi(I_n)$ such that any $\Phi^0(I_n)$-universal $\z$-lattice is $\Phi(I_n)$-universal.  Without loss of generality, we may assume that there is an integer $t_0$ with $0\le t_0\le t-1$ such that $\ell_i=I_{k_i}$ for any $i=1,2,\dots,t_0$, and 
 $\ell_i=I_{k_i} \perp \ell'_i$ for any $i$ with $t_0+1\le i\le t$, where $\ell_i'$ is a $\z$-sublattice of $\ell_i$ such that $\min(\ell'_i )\ge 2$. Since $I_n \not \in \Phi(I_n)$, we have $k_i\le n-1$ for any $i=1,2,\dots,t$. Now, define
$$
L=I_{n-1} \perp \ell'_1\perp \dots\perp \ell'_t. 
$$
Then, the $\z$-lattice $L$ is $\Phi^0(I_n)$-universal and hence $\Phi(I_n)$-universal. 
Since $L$ does not represent $I_n$ itself, it is an isolation of $I_n$. This completes the proof.  
\end{proof}

Though the following lemma is well known, we provide the proof for those who are unfamiliar with this. 

\begin{lem} \label{lem1} Let $p$ be a prime and let $\ell$ be a $\z$-sublattice of $I_n$ with index $p$. Then there are integers $u_2,\dots,u_n$ with $0\le u_2\le \dots\le u_n \le \frac{p}2$ such that 
$$
\ell \simeq I_{u_2,\dots,u_n}(p):=\z(\be_2+u_2\be_1)+\z(\be_3+u_3\be_1)+\dots+\z(\be_n+u_n\be_1)+\z(p\be_1),
$$
where $\{\be_i\}_{i=1}^n$ is an orthonormal basis for $I_n$, that is, $B(\be_i,\be_j)=\delta_{ij}$ for any $i,j$ with $1\le i,j\le n$..  
\end{lem}

\begin{proof}  Without loss of generality, we may assume that  $\min(\ell)\ge 2$.  By Invariant Factor Theorem,  there is a basis $\{\bx_1,\dots,\bx_n\}$ for $I_n$ such that 
$\ell=\z \bx_1+\z \bx_2+\dots+\z \bx_{n-1}+\z p\bx_n$ (see, for example,  81:11 of  \cite{om} ).  For each $i=1,2,\dots,n$, let $\be_i=a_{i1}\bx_1+\dots+a_{in}\bx_n$. From the assumption, we know that $a_{in} \not \equiv 0 \Mod p$. Let $u_i$ be the positive integer less than $p$ such that $a_{in}+u_ia_{1n} \equiv 0 \Mod p$ for any $i=2,3,\dots,n$. Then $\be_i+u_i\be_1 \in \ell$ for any $i=2,3,\dots,n$ and 
$$
I_{u_2,\dots,u_n}(p):=\z(\be_2+u_2\be_1)+\dots+\z(\be_n+u_n\be_1)+\z(p\be_1)\subset \ell \subset I_n.
$$  
Note that $I_{u_2,\dots,u_i,\dots,u_n}(p) \simeq I_{u_2,\dots,(p-u_i),\dots,u_n}(p)$. Hence
if $u_i$ is greater than $\frac {p}2$, then one may replace it with $p-u_i$ so that we may assume that $0\le u_i\le \frac p2$.  
Now, the lemma follows directly from the fact that $[I_n: I_{u_2,\dots,u_n}(p)]=[I_n:\ell]=p$.  \end{proof}

We will frequently use the following well known lemma in the next section. 

\begin{lem} \label{local}  Let $p$ be an odd prime. Let $L_p$ be a quaternary unimodular $\z_p$-lattice and let $\ell_p$ be a binary $\z_p$-lattice. If $dN_p=1$, then $\ell_p \ra N_p$.  If  $dN_p =\Delta_p$, where $\Delta_p$ is a nonsquare unit in $\z_p$, then 
 $$
 \ell_p \nra N_p   \quad \iff \quad d(\q_p\ell_p) =-\Delta_p \ \ \text{and} \ \ H_p(\ell_p)=-1,
 $$ 
where $H_p(\cdot)$ is the Hasse symbol over $\z_p$.  In particular, if $\ell_p$ represents a unit in $\z_p$, then $\ell_p \ra L_p$. 
\end{lem}

\begin{proof} This is a direct consequence of Theorem 1 of \cite{om2}. 
\end{proof}

A $\z$-lattice $R$ is called a {\it root lattice} if it is generated by vectors $\bx$ such that $Q(\bx)=1$ or $2$. It is well-known that any root lattice is isometric to an orthogonal direct sum of indecomposable root lattices, which are $$
I_1,  \quad  A_n \ (n\ge 1), \quad D_n  \  (n\ge 4), \quad \text{and}  \quad E_n  \ (6\le n\le 8).
$$ 
 Conway and Sloane  introduced  in \cite{cs1} a convenient way of describing a $\z$-lattice with small discriminant by using root sublattices of it (see also \cite{co} and Chapter 4 of \cite{cs2}). Throughout this article, we adopt their notation to present a $\z$-lattice with small discriminant.   For those who are unfamiliar with this notation, we briefly introduce Conway and Sloane's notation for some specific case (see also \cite{co}). 
 
Suppose that $L_1, \ldots, L_t$ are integral lattices.  For each $i = 1, \ldots, t$, let $\bx_i$ be a vector in $L_i^\#$.  We define
\begin{equation} \label{gluesum}
L_1\cdots L_t\left[\bx_1\cdots \bx_t\right]: = (L_1 \perp \cdots \perp L_t) + \z(\bx_1 + \cdots + \bx_t).
\end{equation}
Note that this lattice is integral if and only if $Q(\bx_1 + \cdots + \bx_t)$ is an integer.  If $L_i = \z[\bz_i] \cong \langle a \rangle$ for some integer $a$ and $\bx_i = \frac{\bz_i}{m}$, then in the notation $L_1\cdots L_t\left[\bx_1\cdots \bx_t\right]$,  we will use ``$a$" instead of $L_i$ and replace $\bx_i$ by $\frac{1}{m}$.

For $n \geq 1$,  the root lattice $A_n$ is
$$
A_n = \{(a_0, a_1, \ldots, a_n) \in \z^{n+1} : a_0 + \cdots + a_n = 0\},
$$
which is viewed as a sublattice in $\z^{n+1}$.  It is an integral lattice of rank $n$ and discriminant $n + 1$.
Its glue vectors are defined by
$$
\lglue i\rglue_{A_n}=\lglue i\rglue  = \left(\frac{i}{n+1}, \ldots, \frac{i}{n+1}, \frac{-j}{n+1}, \ldots, \frac{-j}{n+1} \right) \in A_n^{\#},
$$
with $j$ components equal to $i/(n+1)$, and $i$ components equal to $-j/(n+1)$, where $i + j = n+1$ and $0 \leq i \leq n$.  
As an example of illustrating \eqref{gluesum}, $A_n\, a \left[i\, \frac{1}{d}\right]$ is the lattice
$$
(A_n\perp \z\bz) + \z\left(\lglue i\rglue + \frac{\bz}{d}\right),
$$
where $\bz$ is a vector orthogonal to $A_n$ such that $Q(\bz) = a$. Another example is $A_1A_18[11\frac12]$, which is the $\z$-lattice
$$
(A_1\perp A_1\perp \z\bz)+\z\left(\lglue1\rglue+\lglue1\rglue+\frac{\bz}2\right),
$$
where $\bz$ is a vector with $Q(\bz)=8$ which is orthogonal to $A_1\perp A_1$. Hence we have
$$
A_1A_18\left[11\frac12\right]=\begin{pmatrix} 2&0&1\\0&2&1\\1&1&3\end{pmatrix}.
$$
 Note that for any $i$ with $0\le i\le n$, $Q(\lglue i\rglue) \le Q(\lglue i\rglue+\bx)$ for any $\bx \in A_n$. 

For $n\ge 4$, the root lattice $D_n$ is 
$$
D_n=\{ (a_1,a_2,\dots,a_n) \in \z^n : a_1+a_2+\dots+a_n \equiv 0 \Mod 2\}.
$$
It is an integral lattice of rank $n$ and discriminant $4$. Its glue vectors are defined by
$$
\begin{array} {rll}
&\lglue 0\rglue_{D_n}=\lglue 0\rglue=(0,0,\dots,0), \quad &Q(\lglue 0\rglue)=0\\
&\lglue 1\rglue_{D_n}=\lglue 1\rglue=\left(\frac12,\frac12,\dots,\frac12\right), \quad &Q(\lglue 1\rglue)=\frac n4\\
&\lglue 2\rglue_{D_n}=\lglue 2\rglue=(0,0,\dots,1), \quad &Q(\lglue 2\rglue)=1\\
&\lglue 3\rglue_{D_n}=\lglue 3\rglue=\left(\frac12,\frac12,\dots,-\frac12\right),  \quad &Q(\lglue 3\rglue)=\frac n4.\\
\end{array}
$$

\section{Isolations of cubic lattices of rank $2$ and $3$}

For an irrecoverable $\z$-lattice $\ell$, recall that
$$
\text{Iso}(\ell)=\min\{\text{rank}(L) : \text{$L$ is an isolation of $\ell$}\}.
$$
As mentioned in the introduction, we have $\text{Iso}(I_1)=3$. In this section, we prove that 
$$
\text{Iso}(I_2)=5 \quad \text{and} \quad  \text{Iso}(3)=6. 
$$

\begin{thm} The quinary $\z$-lattice $\langle 1,2\rangle \perp A_221[1\frac13]$  is an isolation of $I_2$ with minimal rank, and hence  $\text{Iso}(I_2)=5$.   
\end{thm}

\begin{proof} Let $L$ be an isolation of $I_2$.  Then, since $\langle 1,4\rangle \ra L$, there is a $\z$-sublattice $N$ of $L$ such that $L\simeq \langle 1\rangle \perp N$. Furthermore, since $I_2$ is not represented by $L$,  we have $\min(N)\ge 2$. Since $\langle 2,2\rangle \ra L$, we have $\langle 2,2\rangle \ra N$. Furthermore, since 
$$
I_1(3)=\z(\be_1+\be_2)+\z(3\be_2) \simeq A_118\left[1\frac12\right] \ra L \quad \text{and} \quad I_1(3) \nra \langle 1,2,2\rangle,
$$  
we have $\mu_3(N) \le 5$, where $\mu_k(N)$ is the $k$-th successive minimum of $N$ (for this, see \cite{loy}). 
Assume that the rank of $L$ is $4$. Then one may easily check that all quaternary candidates of $L$  and binary sublattices of $I_2$ that are not represented by $L$ are listed in Table 1. Therefore, there does not exist a quaternary isolation of $I_2$. 

\begin{table} [h]
\caption{Quaternary candidates and their exceptions}
\centering
\begin{tabular}{l c}
\hline
Quaternary candidates & An exception \\
\hline\hline
$\langle1,2,2,b\rangle \  \  $                                                       for $2\le b\le 5$          & $I_1(3)$\\
\hline
$\langle 1,2\rangle \perp A_2$                & $I_0(3)$ \\
\hline
$\langle 1,2\rangle \perp A_110[1\frac12]$                & $I_0(5)$ \\
\hline
$\langle 1,2\rangle \perp A_114[1\frac12] $                                   & $I_1(3)$ \\
\hline
$\langle 1,2\rangle \perp  A_118[1\frac12]$                                   & $I_2(5)$ \\
\hline
$\langle 1\rangle \perp A_3 $                 & $I_0(3)$ \\
\hline
$\langle 1\rangle \perp A_1A_18[11\frac12] $                 & $I_2(5)$ \\
\hline
$\langle 1\rangle \perp  A_1A_112[11\frac12]  $                 & $I_0(3)$ \\
\hline
$\langle 1\rangle \perp  A_1A_116[11\frac12]  $                 & $I_0(5)$ \\
\hline
\end{tabular}
\end{table}

Now, we will show that the quinary $\z$-lattice $L=\langle1,2\rangle \perp A_221[1\frac13]$ is an isolation of $I_2$. 
It suffices to show that for any prime $p$ and an integer $k$ with $0\le k\le \frac p2$, the binary $\z$-sublattice 
$$
I_k(p)=\z(e_2+ke_1)+\z(pe_2)=\begin{pmatrix} 1+k^2&kp\\kp&p^2\end{pmatrix}
$$
of $I_2$ with index $p$ is represented by $L$ by Lemma \ref{lem1}.  Since both $I_0(2)=\langle 1,4\rangle$ and $I_1(2)=\langle 2,2\rangle$ are represented by  $L$, we may assume that $p\ge 3$.   To show the existence of a representation, we prove that 
\begin{equation} \label{3}
\widetilde{I_k(p)}:=\begin{pmatrix} 1+k^2&kp\\kp&p^2-2\end{pmatrix} \ra M:= I_1\perp A_221\left[1\frac13\right].
\end{equation}
Note that the class number of the quaternary $\z$-lattice $M$ is one. Furthermore, since $d(\widetilde{I_k(p)})=p^2-2(k^2+1)>0$,  it suffices to show that $\widetilde{I_k(p)}_q$  is  represented by $M_q$ over $\z_q$ for any prime $q$. 

Let $q$ be any prime not contained in $\{2,7,p\}$.  Since $(d(\widetilde{I_k(p)}),k^2+1,q)=1$,  either $\widetilde{I_k(p)}_q$ is unimodular over $\z_q$ or it represents a unit in $\z_q$. Therefore the unimodular 
$\z_q$-lattice $M_q$ represents $\widetilde{I_k(p)}_q$ over $\z_q$ by Lemma \ref{local}. 
If $q=p\ne 7$, then $p^2-2$ is a unit in $\z_q$  and  $M_q$ is unimodular. Assume that $q=7 \ne p$. If $p^2-2(k^2+1) \not \equiv 0 \Mod 7$, then the unimodular $\z_7$-lattice $\widetilde{I_k(p)}_7$ is represented by $M_7 \simeq \langle1,1,-1,-7\rangle$.  
 If $p^2-2(k^2+1)  \equiv 0 \Mod 7$, then $k^2+1$ is a square unit in $\z_q$. Hence   $\widetilde{I_k(p)}_7$ is represented by $M_7$ over $\z_7$. If $q=p=7$, then  $\widetilde{I_k(p)}_7$ is a unimodular $\z_7$-lattice. Hence it is represented by $M_7$ over $\z_7$.  Finally, assume that $q=2$. Note that 
 $$
 d(\widetilde{I_k(p)})=p^2-2(k^2+1) \equiv 5,7\Mod 8 \quad \text{and} \quad \text{$\widetilde{I_k(p)}$ is an odd $\z$-lattice.} 
 $$
Therefore $\widetilde{I_k(p)}_2$ is represented by $M_2 \simeq \langle 3,3,3,5\rangle$ over $\z_2$. This completes the proof. 
\end{proof}

\begin{thm} \label{3main}
The $\z$-lattice $I_2\perp A_3\perp\langle 3\rangle$ of rank $6$ is an isolation of $I_3$ with minimal rank, and hence $\text{Iso}(I_3)=6$.   
\end{thm} 

\begin{proof}  Let $L$ be an isolation of $I_3$. Then $L\simeq I_2\perp N$ for some $\z$-sublattice $N$ of $L$  such that $\min(N) \ge 2$. Since $A_3=I_{1,1}(2)$ is represented by $L$, it is represented by $N$. Therefore if the rank  of $L$ is $5$, then $L\simeq I_2\perp A_3$. However, one may easily check that  $I_{1,1}(3)\simeq A_2 \perp \langle 3\rangle \nra  I_2\perp A_3$.  Hence the rank of $L$ is greater than $5$ and $\mu_4(N) \le 3$. Therefore if the rank of $L$ is $6$, then all possible candidates are 
$$
L \simeq I_2\perp A_3\perp \langle 3\rangle, \quad  I_2\perp A_336\left[1\frac14\right],  \quad \text{or} \quad I_2\perp A_38\left[2\frac12\right].
$$
Note that the third one does not represent $I_{1,1}(3)$. 

Now, we prove that $L=I_2\perp A_3\perp \langle 3\rangle$ is an isolation of $I_3$. To prove this, it suffices to show that for any prime $p$, the $\z$-sublattice
$$
I_{a,b}(p)=\z(\be_1+a\be_3)+\z(\be_2+b\be_3)+\z(p\be_3)  \quad \left(0\le a\le b\le \frac p2\right)
$$
 of $I_3$ with index $p$ is represented by $L$.  If $p\le 5$, then we may directly chcek that $I_{a,b}(p) \ra L$. Hence we assume that $p\ge7$. 
 
 First, assume that $p^2>3(a^2+b^2+1)$. Note that to show $I_{a,b}(p) \ra L$, it suffices to show that
 \begin{equation} \label{3}
 \widetilde{I_{a,b}(p)} :=\begin{pmatrix} 1+a^2&ab&ap\\ap&1+b^2&bp\\ap&bp&p^2-3\end{pmatrix} \ra M:=I_2\perp A_3.
 \end{equation}  
 Since the class number of $M$ is one and $d(\widetilde{I_{a,b}(p)})=p^2-3(a^2+b^2+1)>0$ from the assumption, it suffices to show that  $\widetilde{I_{a,b}(p)}_q$ is represented by $M_q$ over $\z_q$ for any prime $q$.
 
  If $q \ne 2,p$, then $(p^2-3(a^2+b^2+1),a^2+b^2+1,q)=1$. Hence $\widetilde{I_{a,b}(p)}_q$ has a binary unimodular component over $\z_q$. Therefore it is represented by the unimodular $\z_q$-lattice $M_q$.  Assume that $q=p$. Note that $M_q$ is a quinary unimodular $\z_q$-lattice with $dM_q=1$. If $q$ does not divide $a^2+b^2+1$, then  $\widetilde{I_{a,b}(p)}_q$ is unimodular over $\z_q$ and hence it is represented by $M_q$ over $\z_q$. Assume that $q$ divides $a^2+b^2+1$. Then one may easily show that the binary $\z_q$-sublattice of $ \widetilde{I_{a,b}(p)}_q$ 
 $$
 \begin{pmatrix} 1+a^2&ap\\ap&p^2-3\end{pmatrix} \quad \text{or} \quad  \begin{pmatrix} 1+b^2&bp\\bp&p^2-3\end{pmatrix}
 $$
 is unimodular over $\z_q$.  Therefore $\widetilde{I_{a,b}(p)}_q$ is represented by $M_q$ over $\z_q$. 
 
 Finally, assume that $q=2$. First, assume that $a\equiv b\equiv 0\Mod 2$. 
 Then $d(\widetilde{I_{a,b}(p)})=p^2-3(a^2+b^2+1) \equiv 2 \Mod 4$  and $\z_2(\be_1+a\be_3)+\z_2(\be_2+b\be_3) \simeq \langle 1,1\rangle$ or $\langle 1,5\rangle$ over $\z_2$. Hence $\widetilde{I_{a,b}(p)}_2 \simeq \langle 1,\rho,2\epsilon\rangle$ over $\z_2$, where  $\rho \equiv 1 \Mod 4$ and $\epsilon \in \z_2^{\times}$.  Therefore, by Theorem 3 of \cite{om2}, we have 
 $$
 \widetilde{I_{a,b}(p)}_2  \ra M_2\simeq \langle 1,3,3,3,12\rangle \ \ \text{over $\z_2$}.
 $$
 Assume that $a\not \equiv b \Mod 2$. Without loss of generality, we assume that $a$ is odd. Since 
 $$
d(\widetilde{I_{a,b}(p)})=p^2-3(a^2+b^2+1) \equiv 3 \Mod 4 \  \text{and} \ \begin{pmatrix} 1+a^2&ap\\ap&p^2-3\end{pmatrix} \simeq \begin{pmatrix} 2&1\\1&2\end{pmatrix} \ \text{over $\z_2$},
$$
 we have $\widetilde{I_{a,b}(p)}_2 \simeq A_2 \perp \langle \rho\rangle$ over $\z_2$ for some $\rho\equiv 1 \Mod 4$. Therefore it is represented by $M_2$ over $\z_2$.  Finally, assume that $a\equiv b\equiv 1 \Mod 2$. In this case, we have
   $$
d(\widetilde{I_{a,b}(p)})=p^2-3(a^2+b^2+1) \equiv 0 \Mod 8 \  \text{and} \ \begin{pmatrix} 1+a^2&ab\\ab&1+b^2\end{pmatrix} \simeq \begin{pmatrix} 2&1\\1&2\end{pmatrix} \ \text{over $\z_2$}.
$$
Therefore, for some $\alpha \in \z_2$, we have 
$$
(\widetilde{I_{a,b}(p)})_2 \simeq \begin{pmatrix} 2&1\\1&2\end{pmatrix} \perp \langle8\alpha\rangle \ra M_2 \simeq \begin{pmatrix} 2&1\\1&2\end{pmatrix} \perp \langle 1,1,12\rangle. 
$$ 

Now, assume that $a^2+b^2>\frac {p^2}3-1$.  Since
$$
I_{a,b}(p)=\z(\be_1-\be_2+(a-b)\be_3)+\z(\be_1+\be_2+(a+b-p)\be_3)+\z(\be_1+a\be_3),
$$  
we have 
$$
I_{a,b}(p) \simeq \begin{pmatrix} 2+c^2&cd&1+ca\\cd&2+d^2&1+da\\1+ca&1+da&1+a^2\end{pmatrix},
$$
 where $c=a-b$ and $d=a+b-p$. Since $0\le a \le b \le \frac p2$,  one may easily show that $\vert c\vert, \vert d \vert <\frac p2-\sqrt{\frac{p^2}{12}-1}$. To show that $I_{a,b}(p) \ra L$, it suffices to show that 
 $$
 I(c,d):= \begin{pmatrix} 2+c^2&cd&1+ca\\cd&2+d^2&1+da\\1+ca&1+da&a^2-2\end{pmatrix} \ra M=I_2\perp A_3.
 $$
Since $a \ge 2$ and we are assuming that $p\ge 7$, we have
 $$
 d(I(c,d))=p^2-6(c^2+d^2+2) \ge p^2-12\left(\frac{p^2}4+\frac{p^2}{12}-1-p\sqrt{\frac{p^2}{12}-1}+1\right)>0,
 $$
 which implies that  $I(c,d)$ is positive definite.  Assume that $q\ne 2,p$. Note that $I(c,d)_q$ has a binary unimodular component over $\z_q$. Therefore it is represented by $M_q$ over $\z_q$.
   Assume that    $q=p$. If $q$ does not divide $c^2+d^2+2$, then $I(c,d)_q$ has a binary unimodular component over $\z_q$ and hence $I(c,d)$ is represented by $M$ over $\z_q$.  Therefore we may assume that $q$ divides $c^2+d^2+2$. Furthermore, we may also assume that for any $w \in \{c,d\}$,
 $$
 \det\begin{pmatrix} 2+w^2&1+wa\\1+wa&a^2-2\end{pmatrix}=2a^2-2w^2-2aw-5\equiv 0 \Mod q. 
 $$ 
Therefore $q$ divides $c-d$ and also divides $c^2+1$. In this case, since $c^2+2$ is a unit square in $\z_q$, $I(c,d)_q$ is represented by $M_q$ over $\z_q$ by Lemma \ref{local}. 

Finally, assume that $q=2$. Since $c \not \equiv d \Mod 2$, we have
$$
d(I(c,d))=p^2-6(c^2+d^2+2) \equiv 7 \Mod 8.
$$ 
Furthermore, since 
$$
d(\z(\be_1-\be_2+c\be_3)+\z(\be_1+\be_2+d\be_3))=2(c^2+d^2+2) \equiv 6\Mod 8,
$$
we have 
$$
 \z_2(\be_1-\be_2+c\be_3)+\z_2(\be_1+\be_2+d\be_3) \simeq \langle 3,2\rangle \ \text{or} \ \langle 3,10\rangle.
$$
Therefore $I(c,d)_2 \simeq \langle -1,-1,-1\rangle$ over $\z_2$, which is represented by $M_2$ over $\z_2$. This completes the proof.      
\end{proof}

\begin{cor} The quinary $\z$-lattice $L=I_1\perp A_3\perp \langle3\rangle$ is an isolation of $I_2$. 
\end{cor}

\begin{proof} For any proper sublattice $\ell$ of $I_2$, the ternary $\z$-sublattice $I_1\perp \ell$ of $I_3$ is represented by $I_2\perp A_3\perp \langle 3\rangle$ by Theorem \ref{3main}. Hence $\ell$ is represented by $L$. 
\end{proof}

\section{A non-linear lower bound for $\text{Iso}(I_n)$}

In this section, we prove that the minimum rank $\text{Iso}(I_n)$ of isolations of a cubic lattice $I_n$ has a non linear lower bound. More precisely, we show that $\text{Iso}(I_n) \in \Omega(n^{\frac 32-\epsilon})$ for any $\epsilon >0$. 

Let $L$ be a $\z$-lattice. The $\z$-sublattice of $L$ generated by vectors of norm $1$ or $2$ is denoted by $R_L$. Let $\ell$ be a $\z$-sublattice of $L$. If $\ell$ is an indecomposable root sublattice of $L$,  we denote the indecomposable component of $R_L$ containing $\ell$ by $R_L(\ell)$.   For any $\bx \in L$,  we define the projection $\text{Proj}_{\ell}(\bx) \in \q\ell$ of $\bx$ on $\ell$ and the projection $\text{Proj}_{\ell^{\perp}}(\bx)  \in \q \ell^{\perp}$ of $\bx$ on $\ell^{\perp}=\{ \bz \in L : B(\bz,\ell)=0\}$ such that 
$$
\bx=\text{Proj}_{\ell}(\bx)+\text{Proj}_{\ell^{\perp}}(\bx).
$$

For any positive integers $n$ and $k$ with $1< k \le n-1$, we define
$$
\begin{array} {rl}
A_{n,k}:\!\!\!&=A_{n-1}(k^2n)\displaystyle \left[k\frac1n\right] \\
             &=\z(-\be_1+\be_2)+\dots+\z(-\be_{n-1}+\be_n)+\z(-(\be_{n-k+1}+\dots+\be_n)).
             \end{array}
$$
Note that $A_{n,k}$ is a $\z$-sublattice of $I_n$ with index $k$ and  $d(A_{n,k})=k^2$. The vector $-(\be_{n-k+1}+\be_{n-k+2}+\dots+\be_n) \in A_{n,k}$ will always  be denoted by $\bx_{n,k}$. Note that $\bx_{n,k}=\lglue k\rglue+\frac 1n \bx_0$, where $\bx_0=-k(\be_1+\be_2+\dots+\be_n) \in  A_{n-1}^{\perp}$ is a vector such that $Q(\bx_0)=k^2n$.

\begin{lem} \label{noidentity}   Assume that there is a representation $\phi: A_{n,k} \to I_{n-1} \perp L$ for some $\z$-lattice $L$ with $\min(L) \ge 2$.  If $1<k<\sqrt n$, then we have  
$\phi(A_{n,k}) \subset L$. 
\end{lem}
\begin{proof} Since $A_{n-1}$ is not represented by $I_{n-1}$, we have $\phi(A_{n-1}) \subset L$.  Suppose that $\phi(\bx_{n,k})=\bu+\bv$, where $\bu \in I_{n-1}$ and $\bv \in L$. Then clearly, $\text{Proj}_{\phi(A_{n-1})}(\bv)=\phi( \lglue k\rglue)$. Since
$$
\begin{array} {rl} 
k=Q(\phi(\bx_{n,k}))&=Q(\bu)+Q(\bv) \ge Q(\bu)+Q(\text{Proj}_{\phi(A_{n-1})}(\bv))= Q(\bu)+Q( \lglue k\rglue)\\
 &=Q(\bu)+\frac {k(n-k)}n>Q(\bu)+k-1,\\
\end{array}
$$
we have $\bu=0$.  This completes the proof. 
\end{proof} 

\begin{rmk} {\rm  Note that the above lemma does not hold if $\sqrt n \le k$. For example, one may easily check that $A_{16,4}$ is represented by $\langle 1 \rangle \perp A_{15}[4]$, whereas $A_{16,4}$ is not represented by $A_{15}[4]$. }   
\end{rmk}

 \begin{thm}  \label{even} Let $n$ be an integer greater than $15$  and let $k$ be an odd integer such that $1<k<\sqrt n$. Let $L$ be a $\z$-lattice with $\min(L) \ge 2$. If there is a representation $\phi : A_{n,k} \to L$, then $R_L(\phi(A_{n-1}))=A_{m-1}$ for some integer $m\ge n$.  
 \end{thm}
 
 \begin{proof} Since we are assuming that $n\ge 16$, and $R_L(\phi(A_{n-1}))$ is an indecomposable root sublattce of $L$ containing $\phi(A_{n-1})$, we have  either $R_L(\phi(A_{n-1}))=A_{m-1}$ or $R_L(\phi(A_{n-1}))=D_m$ for some integer $m$ with $m\ge n$.   Suppose that $R_L(\phi(A_{n-1}))=D_m$, where  $D_m=\z(\be_{1}-\be_{2})+\dots+\z(\be_{m-1}-\be_m)+\z(\be_{m-1}+\be_m)$.  We assume that $\phi: A_{n,k} \to D_m+\z(\phi(\bx_{n,k}))$. Without loss of generality, we may assume that
 $$
 \begin{array}{ll}
& \phi(A_{n-1})=\z(-\be_{m-n+1}+\be_{m-n+2})+\dots+\z(-\be_{m-1}+\be_m) \quad \text{and}\\
&\phi(\bx_{n,k})= \lglue i\rglue_{D_m}+\bd+\text{Proj}_{D_m^{\perp}}(\phi(\bx_{n,k})),\\
\end{array}
 $$
 where $0\le i\le 3$, $\bd \in D_m$, and $D_m^{\perp}$ is the orthogonal component of $D_m$ in $D_m+\z(\phi(\bx_{n,k}))$. If $i=0$, then $Q(\text{Proj}_{D_m^{\perp}}(\phi(\bx_{n,k})))$ is an integer. 
 Hence $A_{n,k}$ is represented by $D_m$ by Lemma \ref{noidentity}.  Since we are assuming that $k$ is odd, this is a contradiction. If $i=1$ or $3$, then  
 $$
k=Q(\phi(\bx_{n,k})) \ge Q(\lglue i\rglue)+Q(\text{Proj}_{D_m^{\perp}}(\phi(\bx_{n,k}))) \ge \frac m4 >k,
$$
 which is also a contradiction. Finally, assume that $i=2$. Then $D_m+\z(\phi(\bx_{n,k}))=D_m+\z(\be_m+\bx_0)$, where $\bx_0=\text{Proj}_{D_m^{\perp}}(\phi(\bx_{n,k}))$. Since we are assuming that $\min(L)\ge 2$, we have $\bx_0 \ne 0$.    Assume that 
 $$
 \phi(\bx_{n,k})=\sum_{i=1}^ma_i\be_i+t(\be_m+\bx_0),
 $$
 where $\sum_{i=1}^ma_i\be_i \in D_m$ and $t$ is a nonzero integer.  From the assumption,   we have 
 $$
 \sum_{i=1}^ma_i\be_i+t\be_m=\sum_{i=1}^{m-n}a_i\be_i-(a-1)(\be_{m-n+1}+\dots+\be_{m-k})-a(\be_{m-k+1}+\dots+\be_m),
 $$
 for some integers $a_i (1\le i\le m-n)$ and  $a$.  Therefore, we have
 $$
k= Q(\phi(\bx_{n,k}))=Q(\sum_{i=1}^ma_i\be_i+t\be_m)+Q(t\bx_0) \ge \min(k,n-k)+Q(t\bx_0)>k, 
 $$
 which is a contradiction. Therefore we have $R_L(\phi(A_{n-1}))=A_{m-1}$ for some integer $m$ greater than or equal to $n$.  
 \end{proof} 
 
 \begin{rmk} {\rm  Some conditions on $n$ and $k$  cannot be removed.  For example, if $k$ is even, then $A_{n,k}$ is always represented by $D_n$. One may easily check that $A_{8t+3,2t+1}$ is represented by $D_{8t+4}[1]$ for any positive integer $t$. In fact, if we assume that 
{\small $$
 D_{8t+4}[1]=\z(\be_1-\be_2)+\dots+\z(\be_{8t+3}-\be_{8t+4})+\z(\be_{8t+3}+\be_{8t+4})+\z\left(\frac{\be_1+\dots+\be_{8t+4}}2\right),
 $$}
then  $A_{8t+3,2t+1}$ is isometric to
$$
\z(\be_1-\be_2)+\dots+\z(\be_{8t+2}-\be_{8t+3})+\z\left(\frac{\be_1+\dots+\be_{6t+2}-\be_{6t+3}-\dots-\be_{8t+4}}2\right),
$$
which is a sublattice of $D_{8t+4}[1]$.
 }   
\end{rmk}
 
\begin{thm} \label{highermain} 
 Let $n$ be an integer and let $s,k$ be relatively prime odd integers such that $1<s<k<\sqrt{\frac n2}$. Let $L$ be a $\z$-lattice with $\min(L)\ge 2$ such that there are representations  
 $\phi_k : A_{n,k} \to L$ and $\phi_s : A_{n,s} \to L$. Then  we have $R_L(\phi_k(A_{n-1})) \perp R_L(\phi_k(A_{n-1}))$. 
\end{thm}
 \begin{proof}  Recall that  $R_L(\phi_k(A_{n-1}))$ is the indecomposable component of the root sublattice $R_L$ of $L$ containing $\phi_k(A_{n-1})$.
  Suppose on the contrary that $R_L(\phi_k(A_{n-1}))=R_L(\phi_s(A_{n-1}))$. Then by Theorem \ref{even}, there is an integer $m \ge n$ such that  
 $$
 R_L(\phi_k(A_{n-1}))=R_L(\phi_s(A_{n-1}))=A_{m-1}=\z(-\be_1+\be_2)+\dots+\z(-\be_{m-1}+\be_m).
 $$
 Let $\bx_{n,k} \in A_{n,k}$ and $\bx_{n,s} \in A_{n,s}$ be vectors defined before.  Let $\omega=k$ or $\omega=s$. 
 We assume that $\phi_{\omega} : A_{n,\omega} \to A_{m-1}+\z(\phi_{\omega}(\bx_{n,\omega}))$.  By taking a suitable isometry of $A_{m-1}$, if necessary, we may assume that 
 $$
 \phi_{\omega}(A_{n-1})=\z(-\be_{m-n+1}+\be_{m-n+2})+\dots+\z(-\be_{m-1}+\be_{m}).  
 $$
 Then we have
 $$
 \phi_{\omega}(\lglue\omega\rglue)=\frac {\omega}n\left(\be_{m-n+1}+\dots+\be_{m-\omega}\right)+\frac{-(n-\omega)}n\left(\be_{m-\omega+1}+\dots+\be_m\right).
 $$

 First, we prove that 
 \begin{equation} \label{key}
 \text{Proj}_{A_{m-1}}(\phi_{\omega}(\bx_{n,\omega})) \in \lglue\omega\rglue_{A_{m-1}}+A_{m-1}   \quad \text{or} \quad  \lglue m-\omega\rglue_{A_{m-1}}+A_{m-1}.  
 \end{equation}

 Assume that  $\text{Proj}_{A_{m-1}}(\phi_{\omega}(\bx_{n,\omega})) \in \lglue j\rglue+A_{m-1}$ for some $j$ with $0\le j\le m$. Since there is an isometry of $A_{m-1}$ interchanging $j$ and $m-j$, we may assume that $0\le j\le\frac m2$.  Let $\bu= \text{Proj}_{A_{m-1}}(\phi_{\omega}(\bx_{n,\omega})) -\lglue j\rglue \in A_{m-1}$. Since $\phi_{\omega}(A_{n-1}) \subset A_{m-1}$, we have   
 $ \phi_{\omega}(\lglue\omega\rglue)-(\lglue j\rglue+\bu) \in \phi_{\omega}(A_{n-1})^{\perp}$.  First, assume that  $\omega<j\le \frac m2$. Then we have 
 $$
 \omega=Q(\phi_{\omega}(\bx_{n,\omega})) \ge Q(\lglue j\rglue+\bu) \ge Q(\lglue j\rglue)=\frac {j(m-j)}m \ge \omega+1-\frac{(\omega+1)^2}m>\omega,
 $$
 which is a contradiction. 
 
  Now, assume that $0\le j<\omega$.
 Recall that 
 $$
  \phi_{\omega}(\lglue\omega\rglue)=\left(\overbrace{0,\dots,0}^{m-n},\overbrace{ \frac \omega n,\dots\frac \omega n}^{n-\omega},\overbrace{\frac{-(n-\omega)}n,\dots,\frac{-(n-\omega)}n}^{\omega-j},\overbrace{\frac{-(n-\omega)}n,\dots,\frac{-(n-\omega)}n}^{j}\right) 
  $$
  and
  $$
 \lglue j\rglue= \left(\overbrace{\frac jm,\dots,\frac jm}^{m-n},\overbrace{\frac jm,\dots\frac  jm}^{n-\omega},\overbrace{\frac jm,\dots,\frac jm}^{\omega-j},\overbrace{\frac{-(m-j)}m,\dots,\frac{-(m-j)}m}^{j}\right).
 $$
Hence if we define  $\beta=-j/m$ and $\alpha=(\omega m-nj)/mn$, then we have
 $$
 \phi_{\omega}(\lglue\omega\rglue)-\lglue j\rglue=(\overbrace{\beta,\dots,\beta}^{m-n},\overbrace{\alpha,\dots,\alpha}^{n-\omega},\overbrace{-1+\alpha,\dots,-1+\alpha}^{\omega-j},\overbrace{\alpha,\dots,\alpha}^{j}).
 $$
 Therefore there are integers $s_1,\dots,s_{m-n}$ and $s$ such that 
$$
\bu=(\overbrace{s_1,\dots,s_{m-n}}^{m-n},\overbrace{s,\dots,s}^{n-\omega},\overbrace{-1+s,\dots,-1+s}^{\omega-j},\overbrace{s,\dots,s}^{j}),
 $$   
 where $\sum_{i=1}^{m-n} s_i+ns-(\omega-j)=0$.  Hence  
$$
Q(\lglue j\rglue+\bu)=\sum_{i=1}^{m-n}\left(\frac jm+s_i\right)^2+\left(\frac jm+s\right)^2(n-\omega)+\left(\frac jm+s-1\right)^2\omega.
$$
Note that $0\le \frac jm \le \frac{\omega}n \le \frac14$.  If $s \ne 0$, then 
$$
\omega \ge Q(\text{Proj}_{A_{m-1}}(\phi_{\omega}(\bx_{n,\omega})))=Q(\lglue j\rglue+\bu)\ge \left(\frac jm+s\right)^2(n-\omega) \ge \frac 9{16}\left(n-\omega\right) >\omega,
$$ 
which is a contradiction.  Therefore $s=0$ and $\sum_{i=1}^{m-n} s_i=\omega-j>0$.  Consequently, we have 
$$
\begin{array} {ll} 
Q(\lglue j\rglue+\bu)\!\!\!&=\displaystyle \sum_{i=1}^{m-n}\left(\frac jm+s_i\right)^2+\left(\frac jm\right)^2(n-\omega)+\left(\frac jm-1\right)^2\omega \\
&=\displaystyle\sum_{i=1}^{m-n} s_i^2+\frac {2j}m\left(\omega-j\right)+\left(\frac jm\right)^2\left(m-\omega\right)+\left(\frac jm-1\right)^2\omega\\
&=\displaystyle\sum_{i=1}^{m-n} s_i^2+\omega-\frac {j^2}m>\omega,
\end{array} 
$$
which is a contradiction.   

Now, let $\bu_k, \bu_s \in A_{m-1}$ and $\bz_k, \bz_s \in A_{m-1}^{\perp}$ be vectors such that   
$$
\phi_k(\bx_{n,k})=\lglue k\rglue+\bu_k+\frac 1m \bz_k \quad \text{and} \quad \phi_s(\bx_{n,s})=\lglue s\rglue+\bu_s+\frac 1m \bz_s.
$$
Since $B(\lglue k\rglue,\lglue s\rglue)=\frac {s(m-k)}m$, and $B(\phi_k(\bx_{n,k}),\phi_s(\bx_{n,s}))$ is an integer, there is an integer $T$ such that $B(\bz_k,\bz_s)=ksm+Tm^2$. 
For any $\omega \in \{k,s\}$, since 
$$
\omega=Q(\phi_{\omega}(\bx_{n,\omega}))=Q(\lglue \omega\rglue+\bu_k)+Q\left(\frac 1m \bz_{\omega}\right) \ge \frac {\omega(m-\omega)}m+Q\left(\frac 1m \bz_{\omega}\right),
$$
we have $Q(\bz_{\omega}) \le m\omega^2$. Furthermore, since 
$$
0 \le d(\z\bz_k+\z\bz_s)=Q(\bz_k)Q(\bz_s)-B(\bz_k,\bz_s)^2 \le (ksm)^2-(ksm+Tm^2)^2,
$$
and $1<s<k<\sqrt{\frac n2}$ from the assumption, we have $T=0$ and $Q(\bz_{\omega})=m\omega^2$ for any $\omega \in \{k,s\}$.  
Since $k$ and $s$ are relatively prime, there are integers $\tau$ and  $\mu$ such that $\tau k+\mu s=1$.  Choose a vector 
$$
\bu=\lglue 1\rglue-(\tau(\lglue k\rglue+\bu_k)+\mu(\lglue s\rglue+\bu_s)) \in A_{m-1}.
$$
Then we have 
$$
\begin{array} {lll}
Q(\tau \phi_k(\bx_{n,k})+\mu\phi_s(\bx_{n,s})+\bu)\!\!\!&=Q\left(\lglue 1\rglue+\displaystyle\frac 1m\left(\tau \bz_k+\mu \bz_s\right)\right)\\
&=\displaystyle \frac {(m-1)}m+\frac {(\tau^2 mk^2+2\tau\mu ksm+\mu^2ms^2)}{m^2}=1.\\ 
\end{array} 
$$
Therefore $L$ contains a unit vector,  which  is a contradiction. 
  \end{proof} 
\begin{rmk} {\rm Let $L$ be a $\z$-lattice defined by
$$
L=A_{27}+\z\left(\lglue 3\rglue+\frac 1{28}\bx\right)+\z\left(\lglue 5\rglue+\frac1{28}\by\right),
$$
where 
$$
\z \bx+\z \by=\begin{pmatrix} 28\cdot9&-28\cdot13\\-28\cdot13&28\cdot25\end{pmatrix}.
$$ 
Then $L$ is an integral $\z$-lattice of rank $29$  such that $\min(L)=2$ and  $dL=2$. 
Note that both $A_{n,3}$ and $A_{n,5}$ are represented by $L$ for any integer $n$ with  $6 \le n\le 28$.  Hence Theorem \ref{highermain} does not hold for $(s,k,n)=(3,5,n)$ for any integer $n$ with $6\le n\le 28$.} 
\end{rmk}

\begin{thm}  \label{mainlowerb}  Let $n$ be any positive integer greater than $1$  and let $t$ be the number of primes less than $\sqrt{\frac n2}$. Then any isolation of $I_n$ represents 
$$
I_{n-1}\perp D_n\perp \overbrace{A_{n-1}\perp \dots\perp A_{n-1}}^{\text{$(t-1)$-orthogonal sums}}.
$$ 
In particular, we have $2n-1+(n-1)(t-1) \le \text{Iso}(I_n)$ and hence $\text{Iso}(I_n) \in \Omega(n^{\frac32-\epsilon})$ for any $\epsilon>0$.   
\end{thm}

\begin{proof}  Let $p_1=2<p_2<\dots<p_t$ be all primes less than $\sqrt{\frac n2}$. Let $L$ be any isolation of $I_n$. Then clearly, $L\simeq I_{n-1}\perp L_1$ for some $\z$-sublattice $L_1$ of $L$ with $\min(L_1) \ge 2$. 
Since the $\z$-sublattice $A_{n,p_i}=A_{n-1}p_i^2n[p_i\frac1n]$ of $I_n$ is represented by $L$, whereas it is not represented by $I_{n-1}$,  we have $A_{n,p_i} \ra L_1$  for any $i=1,2,\dots,t$.  Note that $A_{n,2}=D_n$. 
 Now, by Theorems \ref{even} and  
 \ref{highermain}, we have 
$$
I_{n-1}\perp D_n\perp \overbrace{A_{n-1}\perp \dots\perp A_{n-1}}^{\text{$(t-1)$-orthogonal sums}} \ra L
$$ 
The theorem  follows directly from this.  \end{proof}

\section{An explicit upper bound for $\text{Iso}(I_n)$}

In this section, we give an explicit upper bound for $\text{Iso}(I_n)$.  Let $L=\z \bx_1+\z \bx_2+\dots+\z \bx_n$ be a $\z$-lattice of rank $n$, and let 
$$
f_L(x_1,x_2,\dots,x_n)=\sum_{i,j=1}^n B(\bx_i,\bx_j)x_ix_j
$$
 be the corresponding quadratic form.  Note that the corresponding quadratic form $f_L$ depends on the choice of the basis for $L$.  Let $h_j$'s and $c_{i,j}$'s be rational numbers such that  
$$
\begin{array} {rl}
f_L(x_1,x_2,\dots,x_n)&=\displaystyle\sum_{i,j=1}^n B(\bx_i,\bx_j)x_ix_j\\
&=h_1(x_1+c_{12}x_2+\dots+c_{1n}x_n)^2\\
&\hskip 1pc +h_2(x_2+c_{23}x_3+\dots+c_{2n}x_n)^2\\
&\hskip 2pc+\dots+h_nx_n^2. \\
\end{array}
$$
We say  the ``ordered''  basis $\{\bx_i\}_{i=1}^n$ for $L$ (or the corresponding quadratic form $f_L$) is {\it Hermite reduced} 
\begin{itemize} 
\item [(i)] if  $\vert c_{ij}\vert \le \frac12$ for any $i,j$ with $1\le i<j\le n$ and 
\item [(ii)]  if, as a quadratic form, 
$$
f^i_L(x_i,x_{i+1},\dots,x_n)=h_i(x_i+c_{i,i+1}x_{i+1}+\dots+c_{in}x_n)^2+\dots+h_nx_n^2
$$
satisfies
$$
\min_{(x_i,\dots,x_n) \in \z^{n-i+1}-\{0\}} f^i_L(x_i,\dots,x_n)=h_i \quad  \text{for any $i=1,2,\dots,n$.}
$$
\end{itemize}
    For each $i=1,2,\dots,n$, the above constant $h_i$ is called  the {\it $i$-th Hermite minimum} of $L$ with respect to the basis $\{\bx_i\}$ (or the {\it $i$-th Hermite minimum} of the corresponding Hermite reduced form $f_L$).  
 Note that  a Hermite reduced basis is different from a Minkowski reduced basis, which is usually adapted (for details, see \cite{ca}). One may easily  check that any $\z$-lattice has a Hermite reduced basis, which is not necessarily unique  in general.  If $a_i$ is an integer satisfying $\vert a_i+c_{i,i+1}\vert \le \frac12$, then from the definition, it satisfies
 $$
 h_i=\min f^i_L(x_i,\dots,x_n) \le f^i_L(a_i,1,0,\dots,0) =h_i(a_i+c_{i,i+1})^2+h_{i+1} \le \frac {h_i}4+h_{i+1}.
 $$
 Therefore we have $\frac{h_i}{h_{i+1}} \le \frac43$ for any $i$ with  $1\le i \le n-1$.   

For a positive integer $n$, let $\mathfrak G(n)$ be the set of all $\z$-lattices of rank $n$ that are represented by a sum of squares $I_m$ for some positive integer $m$.  We define 
$$
g(n)=\min \{N : \ell \ra I_N  \ \ \text{for any $\ell \in \mathfrak G(n)$}\}. 
$$
Larange proved that any positive integer is a sum of four squares and hence $g(1)=4$. In \cite{mo}, Mordell generalized Lagrange's four square theorem by proving that any positive definite integral binary quadratic form is represented by a sum of five squares and hence $g(2)=5$. Ko proved in \cite{ko} that $g(n)=n+3$ for any integer $n=3,4,5$.  In \cite{o1} and \cite{o2}, it was proved that $g(6)=10$.  For some basic properties and upper bounds for $g(n)$ for some large $n$, see  \cite{o3} and \cite{o4}.  As far as the author knows, there is no known exact values of $g(n)$ for any $n\ge 7$. Recently, it was proved in \cite{be} that for any $\epsilon>0$, 
$$
g(n) \in O(e^{(4+2\sqrt2+\epsilon)n}).
$$ 

\begin{thm} \label{upper} For any positive integer $n$, we have 
$$
\text{Iso}(I_n) \le  g(n)+\left(\frac43\right)^n(3n^3-12n^2+48n)+\frac12 n^3-\frac32n^2-47n-1.
$$
\end{thm}

\begin{proof} Let  $L$ be any proper sublattice of $I_n$. Then there is a nonnegative integer $k$ less than $n$ and a sublattice $\ell$ of $L$ with $\min(\ell)\ge 2$ such that $L=I_k \perp \ell$.  If a $\z$-lattice $\mathcal L$ represents all $\z$-lattices of rank less than $n$ whose minimum is greater than $1$, then $I_{n-1} \perp \mathcal L$ is an isolation of $I_n$, and hence $\text{Iso}(I_n) \le n-1+\rank(\mathcal L)$.   

Let $\ell=\z \bx_1+\z \bx_2+\dots+\z \bx_n$ be a $\z$-lattice of rank $n$ such that $\min(\ell) \ge 2$.  We assume that  $\{\bx_i\}_{i=1}^n$ is a Hermite reduced basis for $\ell$.  Assume that the $n$-th Hermite minimum $h_n$ of $\ell$ with respect to this Hermite reduced basis is  greater than or equal to $4$. Define a $\z$-lattice $\widetilde{\ell}=\z \widetilde{\bx_1}+\z \widetilde{\bx_2}+\dots+\z \widetilde{\bx_n}$  such that
$$
B(\widetilde{\bx_i},\widetilde{\bx_j})=\begin{cases} B(\bx_i,\bx_j) \quad &\text{if $(i,j) \ne (n,n)$},\\
                                                                          B(\bx_n,\bx_n)-2 \quad &\text{otherwise.}\end{cases}
$$                                                                          
Then we have
$$
\begin{array} {rl}
f_{\widetilde{\ell}}(x_1,x_2,\dots,x_n) &=f_{\ell}(x_1,x_2,\dots,x_n)-2x^2_n \\
&=h_1(x_1+c_{12}x_2+\dots+c_{1n}x_n)^2\\
&\hskip 1pc +h_2(x_2+c_{23}x_3+\dots+c_{2n}x_n)^2\\
&\hskip 2pc+\dots+(h_n-2)x_n^2. \\
\end{array}
$$
Since we are assuming that $h_n \ge 4$, the $\z$-lattice $\widetilde{\ell}$ is positive definite. 
Now, assume that $(x_1,x_2,\dots,x_n) \in \z^n-\{(0,0,\dots,0)\}$. 
If $x_n=0$, then 
$$
f_{\widetilde{\ell}}(x_1,x_2,\dots,0)=f_{\ell}(x_1,x_2,\dots,0) \ge 2,
$$
and if $x_n \ne0$, then $f_{\widetilde{\ell}}(x_1,x_2,\dots,x_n) \ge  (h_n-2)x_n^2 \ge 2$.  Hence $\min(\widetilde{\ell}) \ge 2$.  Furthermore, by taking a suitable basis for $\ell$, we may assume that there are integers $a_i$'s such that 
$$
f'_{\ell}(x_1,x_2,\dots,x_n)=f'_{\widetilde{\ell}}(x_1,x_2,\dots,x_n)+2(a_1x_1+a_2x_2+\dots+a_nx_n)^2,
$$  
where the corresponding quadratic form $f'_{\widetilde{\ell}}$ is a Hermite reduced form.  Since 
$d(\widetilde{\ell}) <d(\ell)$, by repeating the above process, if necessary,  at most finitely many, we may conclude that there is a positive interger $N$ and integers $a_{ij}$'s such that 
$$
\widetilde{f}_{\ell}(x_1,x_2,\dots,x_n)=2\sum_{i=1}^N (a_{1i}x_1+a_{2i}x_2+\dots+a_{ni}x_n)^2+g(x_1,x_2,\dots,x_n),
$$  
 where $\widetilde{f}_{\ell}$ is a quadratic form corresponding to $\ell$ for a suitable basis, and  $g$ is a Hermite reduced quadratic form  with $\min(g)\ge 2$ such that the $n$-th Hermite minimum $u_n$ of $g$ is less than $4$. From this and the definition of $g(n)$, the quadratic form $\widetilde{f}_{\ell}$ is represented by $2I_{g(n)} \perp g$.  

Now, assume that
$$
\begin{array} {rl}
g(x_1,x_2,\dots,x_n)=&\sum_{i,j=1}^n g_{ij}x_ix_j\\
&=u_1(x_1+d_{12}x_2+\dots+d_{1n}x_n)^2\\
&\hskip 1pc +u_2(x_2+d_{23}x_3+\dots+d_{2n}x_n)^2\\
&\hskip 2pc+\dots+u_nx_n^2, \\
\end{array}
$$
where $u_i$ is the $i$-th Hermite minimum  of $g$ for any integer $i$ with $1\le i\le n$.  Note that
$$
\begin{cases} &2\le g_{ii}=u_1d_{1i}^2+u_2d_{2i}^2+\dots+u_{i-1}d_{i-1,i}^2+u_i \le  \frac14\left(u_1+\dots+u_{i-1}\right)+u_i \\
 &\vert g_{i,j}\vert=\vert u_1d_{1i}d_{2i}+\dots+u_{i-1}d_{i-1,i}d_{i-1,j}+u_id_{ij}\vert \le \frac14\left(u_1+\dots+u_{i-1}\right)+\frac12 u_i, \\
 \end{cases}
 $$
and $0\le u_i <4\cdot \left(\frac43\right)^{n-i}$. Hence if $\mathfrak S(n)$ is the set of all Hermite reduced quadratic forms satisfying the above inequality, then we have
$$
\begin{array} {rl}
\vert \mathfrak S(n)\vert& \le \displaystyle \sum_{i=1}^{n-1}\left[\frac14S_{i-1}+u_i-1+(n-i)\left(2\left(\frac14S_{i-1}+\frac12u_i\right)+1\right)\right]+\frac14S_{n-1}+u_n\\
&=\left(\frac43\right)^n\left(3n^2-12n+48\right)+\frac12n^2-\frac32n-48,
\end{array}
$$
where $S_j=u_1+u_2+\dots+u_j$ for any integer  $j$ with $1\le j\le n$. Since the $\z$-lattice
$$
I_{n-1}\perp 2I_{g(n)} \displaystyle \perp_{g \in \mathfrak S(n)} L_g
$$
represents all proper $\z$-lattices of $I_n$, but not $I_n$ itself, it is an isolation of $I_n$. Now, the theorem follows directly from this.  
\end{proof}

\end{document}